\documentclass[a4paper,12pt]{amsart}

\usepackage{a4wide}
\usepackage[english]{babel}

\usepackage{amsthm,amssymb}

\usepackage[alphabetic]{amsrefs}

\newtheorem{theorem}{Theorem}
\newtheorem{corollary}[theorem]{Corollary}
\newtheorem{question}[theorem]{Question}

\theoremstyle{definition}
\newtheorem{definition}[theorem]{Definition}
\theoremstyle{remark}
\newtheorem{remark}[theorem]{Remark}

\newcommand{\NN}{\mathbb{N}}
\newcommand{\CC}{\mathbb{C}}
\newcommand{\cont}{\mathcal{C}}

\newcommand{\id}{\mathop{\mathrm{id}}}
\newcommand{\aut}{\mathop{\mathrm{Aut}}}
\newcommand{\saut}{\mathop{\mathrm{SAut}}}

\author{Rafael B. Andrist}
\title{Integrable generators of Lie algebras of vector fields on $\mathbb{C}^n$}
\date{21 August 2018}
\address{Rafael B. Andrist \\ Department of Mathematics \\
American University of Beirut \\
Beirut, Lebanon}

\begin{document}

\begin{abstract}
There exist three vector fields with complete polynomial flows on $\mathbb{C}^n$, $n \geq 2$, which generate the Lie algebra generated by all algebraic vector fields on $\mathbb{C}^n$ with complete polynomial flows. In particular, the flows of these vector fields generate a group that acts infinitely transitive. The analogous result holds in the holomorphic setting.
\end{abstract}

\maketitle

\section{Introduction}

We will need the following two notions of flexibility and infinite transitivity introduced by Arzhantsev et al.\ \cite{flexible}, and the so-called density property introduced by Varolin \cites{Varolin1, Varolin2}. These notions describe in a precise way that the group of automorphisms $\aut(X)$ of a complex variety $X$ is ``large''. The subgroup $\saut(X)$ generated by unipotent one-parameter subgroups, i.e.\ complete polynomial flows of polynomial vector fields, is called the \emph{special automorphism group} of $X$. The Lie algebra of all holomorphic vector fields on $X$ will be denoted by $\mathfrak{X}(X)$ and the Lie algebra of all holomorphic vector fields on $X$ preserving a closed form $\omega$ will be denoted by $\mathfrak{X}_\omega(X)$. The group of $\omega$-preserving holomorphic automorphisms is denoted by $\aut_\omega(X)$.

\begin{definition} \hfill
\begin{enumerate}
\item
Let $X$ be a complex algebraic variety. A point $x \in X_{\mathrm{reg}}$ is called \emph{flexible} if the tangent space $T_x X$ is spanned by the orbits of unipotent one-parameter subgroups of $\saut(X)$. The variety $X$ is called \emph{flexible} if every point $x \in X_{\mathrm{reg}}$ is flexible.
\item
Let $X$ be a reduced Stein space. A point $x \in X_{\mathrm{reg}}$ is called \emph{holomorphically flexible} if the completely integrable holomorphic vector fields on $X$ span the tangent space $T_x X$. The space $X$ is called \emph{holomorphically flexible} if every point  $x \in X_{\mathrm{reg}}$ is flexible.
\end{enumerate}
\end{definition}

\begin{definition} \hfill
Let $X$ be a complex manifold and let $G$ be a group. The action of $G$ on $X$ is said to be \emph{infinitely transitive} if it acts $m$-transitively on $X$ for any $m \in \NN$.
\end{definition}

A vector field will is called \emph{complete} or \emph{completely integrable} if its flow map exists for all complex times.

\begin{definition} \hfill
\begin{enumerate}
\item
Let $X$ be a complex algebraic manifold. If the Lie algebra generated by the complete algebraic vector fields on $X$ coincides with the Lie algebra of all algebraic vector fields on $X$, we say that $X$ has the \emph{algebraic density property}.
\item
Let $X$ be a complex manifold. If the Lie algebra generated by the complete holomorphic vector fields on $X$ is dense (w.r.t.\ local uniform convergence) in the Lie algebra of all holomorphic vector fields on $X$, we say that $X$ has the \emph{density property}.
\end{enumerate}
\end{definition}

\begin{definition} \hfill
\begin{enumerate}
\item
Let $X$ be a complex algebraic manifold with an algebraic \emph{volume form} $\omega$, i.e. a nowhere vanishing section of the canonical bundle. If the Lie algebra generated by the complete $\omega$-preserving algebraic vector fields on $X$ coincides with the Lie algebra of all $\omega$-preserving algebraic vector fields on $X$, we say that $(X, \omega)$ has the \emph{algebraic volume density property}.
\item
Let $X$ be a complex manifold with a holomorphic \emph{volume form} $\omega$, i.e. a nowhere vanishing section of the canonical bundle. If the Lie algebra generated by the complete $\omega$-preserving holomorphic vector fields on $X$ is dense (w.r.t.\ local uniform convergence) in the Lie algebra of all $\omega$-preserving holomorphic vector fields on $X$, we say that $(X,\omega)$ has the \emph{volume density property}.
\end{enumerate}
\end{definition}

The main implication of the density property is the so-called Anders{\'e}n--Lempert Theorem:
\begin{theorem}[\cites{AL2,FR,FR-err,Varolin1}]
Let $X$ be a Stein manifold with the density property resp.\ $(X, \omega)$ a Stein manifold with the volume density property. Let $\Omega \subseteq X$ be an open subset (and resp.\ $H^{n-1}_{\mathrm{d}}(\Omega) = 0$ for the holomorphic de Rham cohomology) and $\varphi \colon [0,1] \times \Omega \to X$ be a $\cont^1$-smooth map such that
\begin{enumerate}
\item $\varphi_0 \colon \Omega \to X$ is the natural embedding,
\item $\varphi_t \colon \Omega \to X$ is holomorphic and injective (and resp. $\omega$-preserving) for every $t \in [0,1]$,
\item $\varphi_t(\Omega)$ is a Runge subset of $X$ for every $t \in [0,1]$, and
\end{enumerate}
Then for every $\varepsilon > 0$ and for every compact $K \subset \Omega$ there exists a continuous family $\Phi \colon [0, 1] \to \aut(X)$ resp.\ $\Phi \colon [0, 1] \to \aut_\omega(X)$ such that
$\Phi_0 = \id_X$ and $\| \varphi_t - \Phi_t \|_K < \varepsilon$ for every $t \in [0,1]$.

\smallskip
Moreover, these automorphisms can be chosen to be compositions of flows of completely integrable generators of any dense Lie subalgebra of $\mathfrak{X}(X)$ resp.\ $\mathfrak{X}_\omega(X)$ 
\end{theorem}

\begin{remark}
The following implications are well-known (see \cites{KK-survey})
\[
\begin{split}
\text{algebraic (volume) density property}
\Longrightarrow
\text{(volume) density property} \\
\Longrightarrow 
\text{holomorphic flexibility}
\wedge
\text{holomorphic infinite transitivity}
\end{split}
\]
However, the algebraic density property may not necessarily imply (algebraic) flexibility or (algebraic) infinite transitivity.
The results of \cite{flexible} for irreducible algebraic varieties show that
\[
\text{flexible} \Longleftrightarrow \text{$\saut$ infinitely transitive}
\]
\end{remark}

In \cite{AKZ-transitivity}*{Theorem~2.1} Arzhantsev, Kuyumzhiyan and Zaidenberg have shown that any smooth non-degenerate complex-affine toric variety of dimension at least $2$ is a flexible manifold. More recently they showed \cite{AKZ-generators} that finitely many unipotent subgroups are sufficient in order to generate a subgroup of $\saut$ which acts $m$-transitively for any $m \in \NN$.

In particular, for $X = \CC^n$ they showed that $4$ unipotent subgroups are sufficient. In case of $n = 2$, even $3$ unipotent subgroups are sufficient, see \cite{AKZ-generators}*{Theorem~5.17}.

\bigskip
In this short article we both sharpen and extend this result for $\CC^n, n \geq 2,$ and generalize it further to the holomorphic situation. In fact, $3$ unipotent subgroups are always sufficient. Moreover, the corresponding vector fields can be chosen to generate the whole Lie algebra of volume-preserving algebraic vector fields.
The result also holds in a algebro-holomorphic situation: $3$ complete algebraic vector fields, one of them with necessarily non-algebraic flow, can be chosen such that they generate the Lie algebra of all polynomial vector fields on $\CC^n, n \geq 2$.

\section{Three generators}

\begin{theorem}
The Lie algebra of polynomial vector fields on $\CC^n, n \geq 2,$ is generated by the following three complete polynomial vector fields:
\begin{align}
U &= \frac{\partial}{\partial z_n}\\
V &= \frac{\partial}{\partial z_n} + z_{n}^{3} \frac{\partial}{\partial z_{n-1}} + z_{n} z_{n-1}^{3} \frac{\partial}{\partial z_{n-2}} + \dots + z_n z_{n-1} \cdots z_3  z_2^3 \frac{\partial}{\partial z_1} \\
W &= z_1^2 \cdots z_{n-1}^2 \cdot z_n \frac{\partial}{\partial z_n}
\end{align}
\end{theorem}

\begin{proof}
The completeness of $U$ is obvious. From the ``triangular'' shape of $V$ one can easily deduce that it is a locally nilpotent derivation and hence complete. The flow of $W$ is given by $\varphi_t(z_1, \dots, z_n) = (z_1, \dots, z_{n-1}, \exp(t \cdot z_1^2 \cdots z_{n-1}^2 ) z_n)$ and complete as well.

The polynomial vector fields will be constructed inductively in several steps. It is sufficient to construct all monomial vector fields for each coordinate direction. We first need to take care of low degrees.

\begin{enumerate}
\item By acting $2$-times resp.\ $3$-times with $[U, \cdot]$ on $V$ we obtain
\[
z_n \frac{\partial}{\partial z_{n-1}}, \quad \frac{\partial}{\partial z_{n-1}}
\]
\item We now continue by induction in $k = n-1, \dots, 2$ by acting $2$-times with $\displaystyle \left[\frac{\partial}{\partial z_{k-1}}, \cdot \right]$ on $V$ and obtain
\[
z_n z_{n-1} \cdots z_k \frac{\partial}{\partial z_{k-1}}
\]
and finally
\[
z_n \frac{\partial}{\partial z_{k-1}}, \quad \frac{\partial}{\partial z_{k-1}}
\]
by acting on the previously obtained field with $\displaystyle \left[ \frac{\partial}{\partial z_{\ell}}, \cdot \right]$ \\ once for each $\ell = k, \dots, {n-1} \text{ resp.\ } n$.

\item Note that we can now get all lower degrees of already obtained monomials by forming a Lie bracket with a partial derivatives. The left hand side contains only terms for which we have established they can be generated.

Next, for each $k = 1, \dots, n-1$, we form
\begin{align*}
\left[ \frac{\partial}{\partial z_n}, \left[ W, z_n  \frac{\partial}{\partial z_k} \right]  + \left[ \frac{\partial}{\partial z_k} , W \right] \right] &= z_1^2 \cdots z_{k-1}^2 \cdot z_k^2 \cdot z_{k+1}^2 \cdots z_{n-1}^2  \frac{\partial}{\partial z_k}  \\
\left[ z_n \frac{\partial}{\partial z_k},  z_1 \cdots z_{k-1} \cdot z_k^2 \cdot z_{k+1} \cdots z_{n-1}  \frac{\partial}{\partial z_k}  \right] &=  2 z_1 \cdots z_{k-1} \cdot z_k \cdot z_{k+1} \cdots z_{n-1} \cdot z_n  \frac{\partial}{\partial z_k}
\end{align*}
Moreover, we also want to obtain $\displaystyle z_n^2 \frac{\partial}{\partial z_n}$:
\begin{align*}
\left[ z_n \frac{\partial}{\partial z_k}, z_k z_n \frac{\partial}{\partial z_n} \right] + z_k z_n \frac{\partial}{\partial z_k} &= z_n^2 \frac{\partial}{\partial z_n}
\end{align*}
\item We are now able to obtain all monomials by a two-step inductive process. Let $k, \ell \in \{1, \dots, n\}$ with $k \neq \ell$. Induction in $p_k \in \NN$ yields:
\begin{align*}
\left[ z_k^2 \frac{\partial}{\partial z_k}, z_k^{p_k} \frac{\partial}{\partial z_k} \right] = (2 + p_k) z_k^{p_k+1} \frac{\partial}{\partial z_k}
\end{align*}
For each $k$ and each $p_1, \dots, p_n$ we proceed by induction in all the indices $\ell \neq k$. Let $f$ be a monomial in all other variables but $z_\ell$, with power $p_k$ in $z_k$.
\begin{align*}
\left[ z_\ell^{p_\ell} \frac{\partial}{\partial z_\ell}, z_\ell \frac{\partial}{\partial z_k} \right] &= z_\ell^{p_\ell} \frac{\partial}{\partial z_k} \\
\left[ z_\ell^{p_\ell} \frac{\partial}{\partial z_k}, f(z) \cdot z_k \frac{\partial}{\partial z_k} \right] &= (p_k + 1) \cdot z_\ell^{p_\ell} f(z) \frac{\partial}{\partial z_k} \qedhere
\end{align*}

\end{enumerate}
\end{proof}

\begin{corollary}
The group generated by the flows of $U$, $V$ and $W$ acts infinitely transitive on $\CC^n$.
\end{corollary}
\begin{proof}
This is a direct consequence of the preceding theorem and the Anders\'en--Lempert Theorem.
\end{proof}

\begin{remark}
One should compare this theorem and its corollary also to the result by Wold and the author \cite{freedense} that already $2$ holomorphic automorphisms are sufficient to generate a dense subgroup of the holomorphic automorphism group of $\CC^n$. However, one of these automorphisms was not obtained as a flow of a vector field. The method of proof is not related and cannot be used to further reduce the number of complete vector fields needed for generating the Lie algebra.
\end{remark}

\begin{theorem}
The Lie algebra generated by complete polynomial vector fields on $\CC^n, n \geq 2,$ with polynomial flow is generated by the following three complete vector fields:
\begin{align}
U &= \frac{\partial}{\partial z_n}\\
V' &= \frac{\partial}{\partial z_n} + z_{n}^{5} \frac{\partial}{\partial z_{n-1}} + z_{n}^2 z_{n-1}^{5} \frac{\partial}{\partial z_{n-2}} + \dots + z_n^2 z_{n-1}^2 \cdots z_3^2 z_2^5 \frac{\partial}{\partial z_1} \\
V'' &= \frac{\partial}{\partial z_1} + z_{1}^{5} \frac{\partial}{\partial z_{2}} + z_{1}^2 z_{2}^{5} \frac{\partial}{\partial z_{3}} + \dots + z_1^2 z_{2}^2 \cdots z_{n-2}^2 z_{n-1}^5 \frac{\partial}{\partial z_n} 
\end{align}
\end{theorem}

\begin{proof}
From the ``triangular'' shape of $V'$ and $V''$ one again easily deduces that each of them is a locally nilpotent derivation and induces an algebraic $\CC_+$-action.

The polynomial vector fields will be constructed inductively in several steps. It is sufficient to construct all monomial shear vector fields for each coordinate direction. We first need to take care of low degrees.

\begin{enumerate}
\item By acting $3$-times resp.\ $5$-times with $[U, \cdot]$ on $V'$ we obtain
\[
z_n^2 \frac{\partial}{\partial z_{n-1}}, \quad \frac{\partial}{\partial z_{n-1}}
\]
\item We now continue by induction in $k = n-1, \dots, 2$ by acting $3$-times with $\displaystyle \left[\frac{\partial}{\partial z_{k-1}}, \cdot \right]$ on $V^\prime$ and obtain
\[
z_n^2 z_{n-1}^2 \cdots z_k^2 \frac{\partial}{\partial z_{k-1}}
\]
Note again that we can now get all lower degrees of already obtained monomials by forming a Lie bracket with a previously obtained partial derivative. We obtain in particular $\displaystyle \frac{\partial}{\partial z_{k-1}}$ in the induction step.
\item By acting similarly on $V''$ we obtain also
\[
z_1^2 z_{2}^2 \cdots z_{k-1}^2 \frac{\partial}{\partial z_{k}}
\]
for $k = 2, \dots, n$.
\item For any indices $k, \ell \in \{ 1, \dots, n \}$ with $k \neq \ell$ and any $p \in \NN$ and any polynomial $f$ in all other variables except $z_k$ and $z_\ell$ the following holds:
\[
\left[ \frac{\partial}{\partial z_\ell}, \left[ z_k^2 \frac{\partial}{\partial z_\ell} \left[ z_\ell^2 \frac{\partial}{\partial z_k}, z_k^p \cdot f(z) \frac{\partial}{\partial z_\ell} \right] \right] \right]
= 2 (p+2) z_k^{p+1} \cdot f(z) \frac{\partial}{\partial z_\ell}
\]
By taking linear combinations, this allows us to construct, by induction in $p$, every polynomial in the variables $z_1, \dots, z_n$ except $z_k$ in front of $\displaystyle \frac{\partial}{\partial z_k}$ for each $k$. This is sufficient to obtain all the desired polynomial vector fields according to the results of Anders\'en \cite{AL1}.
\end{enumerate}

\end{proof}

\begin{corollary}
The group generated by the (algebraic) flows of $U$, $V'$ and $V''$ acts infinitely transitive on $\CC^n$, i.e.
The group generated by the three unipotent one-parameter subgroups arising as flows of $U$, $V'$ and $V''$ acts infinitely transitively on $\CC^n$.
\end{corollary}

\begin{proof}
This is again a direct consequence of the preceding theorem and the Anders\'en--Lempert Theorem. By examining the proof of Varolin \cite{Varolin2}*{Theorem~3.1} and by \cite{KK-survey}*{Remark~2.2} in the survey of Kaliman and Kutzschebauch we see that the only potential non-algebraic step is the use of the implicit function theorem. However note that in the proof given by Varolin the implicit function theorem is used only for the flow times. Therefore, the compositions of maps remain polynomial when they arise from polynomial flows.
\end{proof}


Given the initially discussed result \cite{AKZ-generators}*{Theorem~2.1} of Arzhantsev, Kuyumzhiyan and Zaidenberg for finitely generated, infinitely transitive actions on toric varieties and the positive results for the (relative) density property for certain toric varieties \cite{KLL} by Kutzschebauch, Leuenberger and Liendo, the following question arises naturally:
\begin{question}
Can the Lie algebra of polynomial vector fields on a toric variety with the density property be generated by finitely many complete polynomial vector fields, and by how many? 
\end{question}

\end{document}